\documentclass{amsart}
\usepackage{graphicx}
\usepackage{amssymb}

\newtheorem{thm}{Theorem}[section]

\newtheorem{lem}[thm]{Lemma}

\theoremstyle{definition}

\theoremstyle{remark}
\newtheorem{rem}[thm]{Remark}

\begin{document}

\title[Remarks]{Remarks on the quadratic Hessian equation}
\author{Connor Mooney}
\address{Department of Mathematics, UC Irvine}
\email{\tt mooneycr@uci.edu}

\begin{abstract}
We prove that viscosity solutions to the quadratic Hessian equation
$$\sigma_2(D^2u) = 1$$
cannot touch a harmonic function on a minimal surface from below. This can be viewed as a form of strict $2$-convexity. We also prove an a priori interior $C^2$ estimate in terms of the $W^{2,\,p}$ norm, for any $p > 2$. 
Finally, we discuss how these results rule out certain strategies for constructing counterexamples to regularity.
\end{abstract}

\maketitle

\section{Introduction}
The $k$-Hessian equation
\begin{equation}\label{Sigmak}
\sigma_k(D^2u) = 1
\end{equation}
is among the most well-studied fully nonlinear elliptic PDEs. Here $u$ is a function on a domain in $\mathbb{R}^n$, $k \leq n$, and $\sigma_k$ denotes the $k^{\text{th}}$ symmetric polynomial of the eigenvalues. This is on one hand due to the appealing simplicity of its structure, and on the other, due to its appearance in geometric applications.

When $k = 1$, (\ref{Sigmak}) is the Laplace equation and viscosity solutions (see Section \ref{Prelims} for the definition) to (\ref{Sigmak}) are smooth. In contrast, there exist singular viscosity solutions to (\ref{Sigmak}) when $n \geq k \geq 3$, due to Pogorelov \cite{P} and Urbas \cite{U2}. In the case $k = 2$, viscosity solutions are smooth in dimension $n = 2$ \cite{H}, $n = 3$ \cite{WY}, and $n = 4$ \cite{SY2}. The interior regularity question for the quadratic Hessian equation
\begin{equation}\label{Sigma2}
\sigma_2(D^2u) = 1
\end{equation} 
remains open in dimensions five and higher.

The purpose of this note is to prove a couple of results about (\ref{Sigma2}) that both generalize previous results, and show that certain strategies for constructing singular solutions to (\ref{Sigma2}) will not work. The first result is:
\begin{thm}\label{Main1}
Assume that $u$ solves (\ref{Sigma2}) in the viscosity sense on $B_1 \subset \mathbb{R}^n$. Then for any smooth, open, embedded portion $\Sigma$ of a minimal hypersurface in $B_1$ and harmonic function $h$ on $\Sigma$, $u|_{\Sigma}$ cannot touch $h$ from below.
\end{thm}
\noindent Theorem \ref{Main1} can be viewed as a form of strict $2$-convexity. For example, in the case that $u$ is convex, Theorem \ref{Main1} implies that the dimension of the agreement set between $u$ and a supporting linear function is at most $n-2$, recovering (with a different proof) the main result from \cite{M5}. More generally, minimal surfaces are natural objects in the study of the $\sigma_2$ equation because the level sets of a solution are mean-convex. Theorem \ref{Main1} proves that the level sets are in a sense strictly mean-convex. The analogous statement for (\ref{Sigmak}) when $n \geq k \geq 3$, namely, that the level sets are strictly $k-1$-convex, is false in view of the singular solutions (\cite{P},\,\cite{U2}) mentioned above.

The second result is:
\begin{thm}\label{Main2}
Assume that $u$ is a smooth solution to (\ref{Sigma2}) in $B_1 \subset \mathbb{R}^n$. Then for any $p > 2$ we have 
\begin{equation}\label{C2}
\|D^2u\|_{L^{\infty}(B_{1/2})} \leq C\left(n,\,p,\,\|D^2u\|_{L^p(B_1)}\right).
\end{equation}
\end{thm}
\noindent The same estimate was obtained for $p > n-1$ by Urbas \cite{U1} and for $p = 3$ by Shankar and Yuan, see Remark 2.2 in \cite{SY1}.

Theorems \ref{Main1} and \ref{Main2} grew out of attempts to construct singular viscosity solutions to (\ref{Sigma2}). 
One motivation for Theorem \ref{Main1} is that the distance function from a minimal surface is a viscosity supersolution to (\ref{Sigma2}) (see Section \ref{Discussion}). An interesting feature of this function is that its graph has an ``upwards corner" more typical of subsolution behavior in elliptic PDEs. One might thus hope to use the maximum principle to build a singular solution that is trapped between the distance function from a minimal surface and some subsolution that also vanishes and has an ``upwards singularity" along the minimal surface. However, Theorem \ref{Main1} rules out this approach. Theorem \ref{Main1} is also motivated by the Pogorelov-type interior $C^2$ estimate for (\ref{Sigmak}) due to Chou and Wang \cite{CW} (see Theorem \ref{ChouWang} below), which says that solutions to (\ref{Sigmak}) that are strictly $k$-convex in an appropriate sense are smooth. The approach to the regularity problem for (\ref{Sigma2}) suggested by the Chou-Wang estimate (that is, to understand the correct notion of strict $2$-convexity and then establish it) seems to have potential. For example, in \cite{M5}, the main result is combined with the estimate of Chou-Wang to show that convex viscosity solutions to (\ref{Sigma2}) are smooth. See Section \ref{CWSec} and the end of Section \ref{SecondProof} for further discussion.

\begin{rem}
The regularity of convex, resp. semiconvex solutions to (\ref{Sigma2}) was investigated using different methods by Guan and Qiu \cite{GQ} and McGonagle, Song, and Yuan \cite{MSY}, resp. Shankar and Yuan (\cite{SY3}, \cite{SY1}).
\end{rem}

Other singularity models are suggested by blowup analysis. Since viscosity solutions to (\ref{Sigmak}) are locally Lipschitz (see Theorem \ref{GradEst} and the remark after it), blowup by one-homogeneous rescaling yields a global Lipschitz viscosity solution to
\begin{equation}\label{Homogeneous2}
\sigma_2(D^2u) = 0.
\end{equation}
Simple examples of global Lipschitz solutions to (\ref{Homogeneous2}) are those of the form 
$$u = w(x'),$$
where $x' \in \mathbb{R}^m$ for some $m < n$ and $w$ is one-homogeneous and smooth away from the origin. The cases $m \leq 3$ all turn out to be the same, and by Theorem \ref{Main1}, the simplest way that such singularity models could arise is ruled out. On the other hand, when $m \geq 3$, such functions are in $W^{2,\,p}$ with $p > 2$, so Theorem \ref{Main2} suggests that these models are unlikely to arise. Combined with the results in \cite{SY1} and \cite{SY2}, these observations indicate that if singular solutions to (\ref{Sigma2}) exist, they are probably fairly complicated. See Section \ref{Discussion} for further discussion.

The proof of Theorem \ref{Main1} is based on the construction of a barrier that agrees with $h$ on $\Sigma$ and has gradient blowing up close to $\Sigma$. The result is obtained by combining this barrier with an interior gradient estimate for (\ref{Sigmak}) due to Trudinger \cite{T}. The proof of Theorem \ref{Main2} is based on Krylov-Safonov and Alexandrov-Bakelman-Pucci methods that have previously been used to prove interior $C^2$ estimates for concave, uniformly elliptic equations. We follow the general idea presented in the proof of Theorem 4.8 in the book of Caffarelli and Cabr\'{e} \cite{CC}. Another method that works, mentioned in Remark 2.2 from \cite{SY1}, is to adapt the proof of Theorem 9.20 in the book \cite{GT} of Gilbarg and Trudinger. However, we feel that the method we present connects more transparently with the Riemannian geometry underlying the equation (\ref{Sigma2}), and it also sheds light on the role of strict $k$-convexity in the Chou-Wang estimate mentioned above.

The paper is organized as follows. In Section \ref{Prelims} we recall a few useful facts about the $k$-Hessian equation. In Section \ref{FirstProof} we prove Theorem \ref{Main1}. In Section \ref{SecondProof} we prove Theorem \ref{Main2}. Finally, in Section \ref{Discussion} we discuss in more detail the motivations for these results, and the implications the results have for the $\sigma_2$ regularity problem.

\section*{Acknowledgements}
This work was supported by a Simons Fellowship, a Sloan Fellowship, and NSF grant DMS-2143668.


\section{Preliminaries}\label{Prelims}
In this section we recall a few well-known facts about the $k$-Hessian equation.

\subsection{Structure}
The G\r{a}rding cone of $k$-convex matrices is defined by
$$\Gamma_k := \{M \in \text{Sym}_{n \times n} : \sigma_i(M) > 0 \text{ for all } i \leq k\}.$$
We say that a $C^2$ function $\varphi$ is $k$-convex if 
$$D^2\varphi \in \overline{\Gamma_k}.$$
The function $\sigma_k$ is elliptic in $\Gamma_k$, uniformly elliptic on compact subsets of $\Gamma_k$, and has convex level sets in $\Gamma_k$. As a consequence of this, the Evans-Krylov theorem (\cite{E}, \cite{K}), and Schauder estimates, $C^2$ $k$-convex solutions to (\ref{Sigmak}) are smooth and enjoy interior derivative estimates of all orders depending on their $C^2$ norm.

In the case $k = 2$, the equation (\ref{Sigmak}) can be written
\begin{equation}\label{Sigma2'}
F(D^2u) := \sigma_2(D^2u) = \frac{1}{2}[(\Delta u)^2 - |D^2u|^2] = 1.
\end{equation}
The linearized operator at a smooth solution $u$ to (\ref{Sigma2'}) is $F_{ij}(D^2u)\partial_i\partial_j$, where
$$F_{ij}(D^2u) := \Delta u\,\delta_{ij} - u_{ij}.$$
We let $DF$ denote the matrix with entries $F_{ij}$. The matrix $DF$ satisfies
\begin{equation}\label{BasicIneqs}
\det DF \geq c(n)(\Delta u)^{n-2},
\end{equation}
see e.g. Corollary 2.1 in \cite{SY2}.

\subsection{Classical Solvability of the Dirichlet Problem}

In \cite{CNS}, Caffarelli, Nirenberg, and Spruck established:
\begin{thm}\label{Classical}
Let $R > 0$ and let $g \in C^{\infty}(\partial B_R)$. There exists a unique $k$-convex solution $u \in C^{\infty}\left(\overline{B_R}\right)$ to
$$\sigma_k(D^2u) = 1 \text{ in } B_R, \quad u|_{\partial B_R} = g.$$
\end{thm}
The domain $B_R$ can in fact be replaced by any smooth, bounded domain whose boundary has second fundamental form in $\Gamma_{k-1}$.

\subsection{Viscosity Solutions}
A continuous function $u$ on a domain $\Omega \subset \mathbb{R}^n$ is a viscosity super (resp. sub) solution to (\ref{Sigmak}) if, for any $C^2$ $k$-convex function $\varphi$ that touches $u$ from below (resp. above) at some point $x_0 \in \Omega$ (that is, $\varphi \leq (\text{resp. }\geq) \,u$ in $\Omega$ and $\varphi(x_0) = u(x_0)$), we have
$$\sigma_k(D^2\varphi)(x_0) \leq \quad (\text{resp. }\geq) \quad 1.$$
We say that $u$ is a viscosity solution to (\ref{Sigmak}) if it is both a viscosity subsolution and supersolution to (\ref{Sigmak}). For $C^2$ $k$-convex functions, the notion of classical and viscosity solution coincide. The property of being a viscosity solution is preserved under local uniform convergence.

\begin{rem}\label{Viscosity1}
The definition of viscosity subsolution to (\ref{Sigmak}) does not change if we allow arbitrary (not just $k$-convex) $C^2$ test functions from above. However, it is important that $k$-convex test functions are used from below in the definition of viscosity supersolution.
\end{rem}

\begin{rem}\label{Viscosity2}
If $u$ is a viscosity solution to (\ref{Sigmak}) in $\Omega \subset \mathbb{R}^n$, then for any ball $B \subset \subset \Omega$, the function $u$ is the uniform limit of smooth viscosity solutions to (\ref{Sigmak}) in $B$ obtained by taking smooth approximations of $u|_{\partial B}$ and the corresponding solutions to the Dirichlet problem from Theorem \ref{Classical}.
\end{rem}

\subsection{Interior Gradient Estimate}
In \cite{T} Trudinger proved the following:
\begin{thm}\label{GradEst}
If $u$ is a smooth $k$-convex solution to (\ref{Sigmak}) in $B_R$, then
$$\|\nabla u\|_{L^{\infty}(B_{R/2})} \leq C(n,\,k)R^{-1}\|u\|_{L^{\infty}(B_R)}.$$
\end{thm}
A consequence of this estimate and Remark \ref{Viscosity2} is that viscosity solutions to (\ref{Sigmak}) are locally Lipschitz.

\subsection{Chou-Wang Estimate}\label{CWSec}
In \cite{CW} Chou and Wang proved the following Pogorelov-type interior $C^2$ estimate.
\begin{thm}\label{ChouWang}
If $\Omega \subset \mathbb{R}^n$ is a bounded domain, $u \in C^{\infty}(\Omega) \cap C\left(\overline{\Omega}\right)$ is a $k$-convex solution to (\ref{Sigmak}), and $w \in C^{\infty}(\Omega) \cap C\left(\overline{\Omega}\right)$ satisfies
$$Lw \geq 0, \quad w|_{\partial \Omega} = u, \quad w > u,$$
then
$$\sup_{\Omega} [(w-u)^4|D^2u|] \leq C(n,\,k,\,\|\nabla u\|_{L^{\infty}(\Omega)},\, \|\nabla w\|_{L^{\infty}(\Omega)}).$$
Here $L$ denotes the linearized operator $(\sigma_k)_{ij}(D^2u)\partial_i\partial_j$.
\end{thm}

\begin{rem}
In the statement in \cite{CW}, $w$ is actually assumed to be $k$-convex. However, an inspection of the proof reveals that the (less restrictive) condition $Lw \geq 0$ suffices. Note that by the maximum principle, the largest function $w$ satisfying the conditions of Theorem \ref{ChouWang} is the solution to 
$$Lw = 0,\, w|_{\partial \Omega} = u.$$
Note also that $Lu = k$, by the $k$-homogeneity of $\sigma_k$. Theorem \ref{ChouWang} thus says that once we solve
$$Lv = 1, \quad v|_{\partial \Omega} = 0,$$
we get an interior $C^2$ estimate for $u$ depending on the size of $1/|v|$. In other words, interior $C^2$ estimates for (\ref{Sigmak}) arise from a strong maximum principle for the linearized operator. This idea has been explored deeply in the context of the Monge-Amp\`{e}re equation (the case $k = n$), see the discussion in Section \ref{SecondProof}.
\end{rem}

\begin{rem}
Theorem \ref{ChouWang} implies the smoothness of viscosity solutions to (\ref{Sigmak}) near points of pointwise twice differentiability. Indeed, if $u$ solves (\ref{Sigmak}) in a domain $U$ and is pointwise twice differentiable at $x_0 \in U$, i.e. there exists a quadratic polynomial $P$ such that
$$u(x) = P(x) + o(|x-x_0|^2),$$
then $D^2P \in \Gamma_k$. Taking $\epsilon > 0$, then $\epsilon' > 0$ sufficiently small, we can ensure that
$$w := P - \epsilon|x-x_0|^2 + \epsilon'$$
is $k$-convex and satisfies that $w > u,\, w|_{\partial \Omega} = u$ on some small domain $\Omega \subset \subset U$ containing $x_0$. Taking a sequence of smooth approximations to $u$ as in Remark \ref{Viscosity2} and applying Theorem \ref{ChouWang}, we see that the approximations have locally uniformly bounded derivatives of all orders in $\Omega$. This result (smoothness of solutions to elliptic PDEs near points of twice differentiability) of course holds in far greater generality by Savin's small perturbations theorem \cite{S}, which was invoked at an important point in the proof of regularity for (\ref{Sigma2}) in four dimensions \cite{SY2}.
\end{rem}

\section{Proof of Theorem \ref{Main1}}\label{FirstProof}
Before proving Theorem \ref{Main1} we construct a barrier. Let
$$f(s) := s\sqrt{|\log{s}|},\, s \in (0,\,1/2).$$
The function $f$ has unbounded derivative near $0$, and
\begin{equation}\label{FBounds}
ff'' = -1/2 + o(1) \text{ as } s \rightarrow 0^+, \quad sf'f'' = -1/2 + o(1) \text{ as } s \rightarrow 0^+.
\end{equation}

Let $\Sigma$ be a smooth embedded minimal hypersurface in $\mathbb{R}^n$, with $0 \in \Sigma$ and $B_4 \cap \partial \Sigma = \emptyset$. Then there is a tubular neighborhood $\mathcal{N}$ of $\Sigma$ in $B_{3}$ in which the signed distance function $d$ from $\Sigma$ is smooth, the nearest point projection to $\Sigma$ is well-defined and smooth, and the unit normal $\nu$ to $\Sigma$ can be extended smoothly to be constant along lines perpendicular to $\Sigma$. By minimality, we have
\begin{equation}\label{LaplaceD}
|\Delta d| \leq C_0|d|
\end{equation} 
in $\mathcal{N}$ for some $C_0 > 0$. 

Let $h$ be a harmonic function on $\Sigma$. Extend it to $\mathcal{N}$ by taking it to be constant along lines perpendicular to $\Sigma$. Then
\begin{equation}\label{LaplaceH}
|\Delta h| \leq C_1 |d|
\end{equation}
in $\mathcal{N}$ for some $C_1 > 0$. 

Let $\varphi \in C^{\infty}_0(\Sigma \cap B_{1})$, and extend it to $\mathcal{N}$ by taking it to be constant along lines perpendicular to $\Sigma$. 

Finally, fix $K \in \mathbb{R}$, and define $\psi_{\delta}$ on $\mathcal{N}$ by
$$\psi_{\delta} := \delta \varphi f(d) + h + Kd.$$
We claim that there exists $\tau(K,\,C_0,\,C_1) > 0$ small such that
\begin{equation}\label{KeyIneq}
\sigma_2(D^2\psi_{\delta}) < 1
\end{equation}
in $B_{2} \cap \{0 < d < \tau\}$ for all $\delta > 0$ small. 

To verify this, note first that
$$2\sigma_2(D^2\psi_{\delta}) = (\Delta \psi_{\delta})^2 - |D^2\psi_{\delta}|^2 \leq (\Delta \psi_{\delta})^2 - (\psi_{\delta})_{\nu\nu}^2.$$
By construction we have $\nabla d \cdot \nabla \varphi = 0$, so 
$$\Delta \psi_{\delta} = \delta(\varphi f'' + \varphi f'\Delta d + f\Delta \varphi) + \Delta h + K\Delta d.$$
Since $(\psi_{\delta})_{\nu\nu} = \delta \varphi f''$ we conclude that
$$2\sigma_2(D^2\psi_{\delta}) \leq (\Delta h + K\Delta d)^2 + \delta E,$$
where $E$ consists of terms that are products of $\varphi$ and its derivatives, possibly $\delta$, and expressions of the form
\begin{align*}
f'^2(\Delta d)^2,\, f^2,\, f''f'\Delta d,\, f''f,\, f''(\Delta h + K\Delta d),\, \\ 
ff'\Delta d,\, f'\Delta d(\Delta h + K\Delta d),\, \text{ and } f(\Delta h + Kf\Delta d).
 \end{align*}
Using (\ref{LaplaceD}) and (\ref{LaplaceH}), we can fix $\tau$ sufficiently small that
$$(\Delta h + K\Delta d)^2 < 1 \text{ in } \{0 < d < \tau\} \cap B_2.$$
Using (\ref{FBounds}), (\ref{LaplaceD}), and (\ref{LaplaceH}), it is easy to see that all the terms in $E$ are bounded in $\{0 < d < \tau\} \cap B_2$, so for $\delta$ small we have $\delta E < 1$ and the desired estimate (\ref{KeyIneq}) follows.

\begin{proof}[{\bf Proof of Theorem \ref{Main1}}]
Assume by way of contradiction that $h$ touches $u|_{\Sigma}$ from above at a point in $\Sigma$. After translating and rescaling quadratically we may assume that $u$ is a viscosity solution of (\ref{Sigma2}) in $B_4$, that $0 \in \Sigma$, that $B_4 \cap \partial \Sigma = \emptyset,$ and that $h$ touches $u|_{\Sigma}$ from above at $0$. Take $\varphi \in C^{\infty}_0(\Sigma \cap B_1)$ such that $\varphi(0) < 0$. Let $K$ be twice the Lipschitz constant for $u$ in $B_3$ (note that $K < \infty$, see Theorem \ref{GradEst} and the comment that follows). Finally, take $\psi_{\delta}$ and $\tau$ as above.

By the choice of $K$, we have for $\delta > 0$ small that $\psi_{\delta} \geq u$ on $\partial (\{0 < d < \tau\} \cap B_2)$. Taking $\delta > 0$ smaller if necessary we also have (\ref{KeyIneq}). We conclude that $\psi_{\delta} \geq u$ in $\{0 < d < \tau\} \cap B_2$ (see Remark \ref{Viscosity1}). However, it is clear from the definition of $f$ and the fact that $\varphi(0) < 0$ that there exist nonzero points $x \in \{0 < d < \tau\} \cap B_2$ such that
$$|x|^{-1}(\psi_{\delta}(x) - \psi_{\delta}(0)) \rightarrow -\infty.$$
Combining this with $u \leq \psi_{\delta},\, u(0) = \psi_{\delta}(0)$, we contradict that $u$ is locally Lipschitz.
\end{proof}

\begin{rem}
The barrier shares features with the (simpler) one used to prove a form of strict convexity in Lemma 2.1 of \cite{MS}.
\end{rem}

\begin{rem}
It is clear from the proof that the assumption that $u$ solves (\ref{Sigma2}) in the viscosity sense can be relaxed to the assumption that $u$ is a Lipschitz viscosity subsolution to (\ref{Sigma2}).
\end{rem}

\section{Proof of Theorem \ref{Main2}}\label{SecondProof}
In this section we assume that $u$ is a smooth solution to $\sigma_2(D^2u) = 1$ in $B_1$. We fix $p > 2$ and we let $K := \|\Delta u\|_{L^{p}(B_1)}$. 

\begin{lem}\label{BasicMeasure}
There exist $c_0(n,\,K,\,p),\, \Lambda(n,\,p) > 0$ such that if
$$\text{tr}(DF\,D^2\varphi) \leq 0 \text{ and } \varphi \geq 0 \text{ in } B_r(x) \subset B_1,$$
then
$$|\{\varphi \leq 2\varphi(x)\} \cap B_r(x)| \geq c_0r^{\Lambda}.$$
\end{lem}
\begin{proof}
Assume without loss of generality that $x = 0$ and $\varphi(0) = 1/16$. Slide the paraboloids of Hessian $-r^{-2}I$ and vertex in $\overline{B_{r/4}}$ up until they touch $\varphi$. Then the contact happens on a compact set $E \subset \overline{B_{3r/4}} \cap \{\varphi \leq 1/8\}$. At a contact point $x$, the corresponding vertex $y$ is given by
$$y(x) = x + r^2\nabla \varphi(x).$$
The differential of this map is given by
$$Dy(x) = I + r^2D^2\varphi(x).$$
At $x$ it is obvious that $D^2\varphi \geq -r^{-2}I$, hence $Dy(x) \geq 0$. Using AGM, the equation for $\varphi$, and (\ref{BasicIneqs}) we conclude that
$$|\det Dy| = \frac{\det(DF + r^2DF D^2\varphi)}{\det DF} \leq \frac{(\text{tr}DF)^n}{\det DF} \leq C(n)(\Delta u)^2.$$
Using the area formula and H\"{o}lder's inequality, we conclude that
$$|B_{r/4}| \leq \int_{E} |\det Dy| \leq C(n)\int_{E} (\Delta u)^2 \leq C(n,K)|E|^{1-2/p}.$$
Since $\varphi \leq 1/8$ on $E$, the inequality follows.
\end{proof}

Below we let 
$$\psi := \Delta u,$$
and we recall that 
$$\text{tr}(DFD^2\psi) \geq 0$$ 
by the convexity of $\{F= 1\} \cap \Gamma_2$.

\begin{lem}\label{RadBound}
There exist $C_1(n,\,K,\,p),\,\delta(n,\,p) > 0$ such that if $B_r(x) \subset B_1,\, \psi(x) = T > 0,$ and $B_r(x) \subset \{\psi < 3T/2\}$, then
$$r < \frac{1}{2}C_1T^{-\delta}.$$
\end{lem}
\begin{proof}
Applying Lemma \ref{BasicMeasure} to $3T/2-\psi$ in $B_r(x)$, we get
$$|\{\psi \geq T/2\} \cap B_r(x)| \geq c_0r^{\Lambda}.$$
On the other hand, we have
$$(T/2)^{p}|{\psi \geq T/2}| \leq \int \psi^{p} = K^p.$$
\end{proof}

\begin{proof}[{\bf Proof of Theorem \ref{Main2}}]
Choose $M(n,\,K,\,p)$ so large that
$$\sum_{k \geq 0} r_k < 1/4, \quad r_k := C_1M^{-\delta}(3/2)^{-\delta k}.$$
Here $C_1,\,\delta$ are the constants from the statement of Lemma \ref{RadBound}.
Assume by way of contradiction that $\psi(x_0) \geq M$ for some $x_0 \in B_{1/2}$. We claim that for all $k \geq 1$, there exists $x_k$ such that $|x_k - x_0| \leq \sum_{i <k} r_i$ (in particular, $x_k \in B_{3/4}$) and $\psi(x_k) \geq (3/2)^kM$, contradicting the continuity of $\psi$. This comes from inductively applying Lemma \ref{RadBound} in $B_{r_k}(x_k),\, k \geq 0$.
\end{proof}

\begin{rem}
This argument shows, more precisely, that
$$\|\Delta u\|_{L^{\infty}(B_{1/2})} \leq C(n,\,p)\|\Delta u\|_{L^p(B_1)}^{\frac{p}{p-2}}.$$
The bound on the full Hessian follows from the equation: 
$$\sqrt{2 + |D^2u|^2} = \Delta u.$$
\end{rem}

We conclude the section with a few remarks concerning the geometry of the linearized equation
$$F_{ij}(D^2u)\varphi_{ij} = 0.$$
It is not hard to check that $\partial_i(F_{ij}(D^2u)) = 0$ for all $j$, so the equation can be rewritten in divergence form as
$$\partial_i(F_{ij}(D^2u)\varphi_j) = 0,$$
or
\begin{equation}\label{Geometric}
\Delta_g\varphi = 0.
\end{equation}
Here $\Delta_g$ is the Laplace-Beltrami operator $\frac{1}{\sqrt{\det g}}\partial_i(\sqrt{\det g}g^{ij}\partial_j)$ with respect to the metric
$$g := (\det DF)^{\frac{1}{n-2}}(DF)^{-1}.$$
(We assume $n \geq 3$. The case $n = 2$ is the Monge-Amp\`{e}re equation, which we discuss below.) Since $g$ has volume element 
$$\sqrt{\det g} = (\det DF)^{\frac{1}{n-2}} \geq \Delta u \quad (\text{by inequality \ref{BasicIneqs}}),$$ 
the requirement that $\psi = \Delta u$ is slightly better than $L^2$ corresponds to $\psi$ being slightly better than integrable with respect to the metric $g$. From this point of view, estimates in terms of $\|\psi\|_{L^p},\, p > 2$ are perhaps somewhat natural.

However, the form (\ref{Geometric}) of the equation suggests that one should not work with euclidean balls, resp. paraboloids as we do in the proof of Theorem \ref{Main2}, but instead with balls taken with respect to the metric $g$, resp. solutions to the Dirichlet problem 
$$\Delta_gP = \text{const.}$$ 
on these balls with linear boundary data. This perspective is useful in the study of the Monge-Amp\`{e}re equation (the case $k = n$ of (\ref{Sigmak})). In that case, the metric is $D^2u$, the correct notion of a ball is a sublevel set of $u + \text{linear}$, and $u + \text{linear}$ is the correct notion of paraboloid. Moreover, the Monge-Amp\`{e}re equation is affine invariant, which gives a way of renormalizing these balls. Using these observations, Caffarelli and Gutierrez \cite{CG} extended the Krylov-Safonov theory \cite{KS} of uniformly elliptic equations to the linearized Monge-Amp\`{e}re equation, when $u$ is strictly convex. In particular, if $\det D^2u = 1$, $u$ is strictly convex, and $\varphi \geq 0$ satisfies
$$u^{ij}\varphi_{ij} \geq 0,$$
then similar arguments to those in the proof of Theorem \ref{Main2} using the correct balls and paraboloids give an interior estimate for $\|\varphi\|_{L^{\infty}}$ in terms of $\|\varphi\|_{L^1}$. As observed e.g. by Maldonado in \cite{Ma}, applying this to $\varphi = \Delta u$ gives a different proof of Pogorelov's interior $C^2$ estimate for strictly convex solutions to the Monge-Amp\`{e}re equation (\cite{P}, also the case $k = n$ of Theorem \ref{ChouWang}). 

\begin{rem}
A streamlined exposition of this proof of Pogorelov's estimate, as well as another short proof using the ABP method that exploits Yuan's Jacobi inequality \cite{Y}, are given in the forthcoming book \cite{M1}.
\end{rem}

In the theory of the linearized Monge-Amp\`{e}re equation, the strict convexity of the underlying function $u$ allows for geometric control of the correct balls and paraboloids. More precisely, they behave like Euclidean balls and paraboloids, up to affine renormalizing transformations that can be effectively estimated. One might hope for a proof of Theorem \ref{ChouWang} along similar lines, using strict $k$-convexity to say that correctly chosen balls and paraboloids (solutions to the linearized equation with constant RHS on these balls, with linear boundary data) behave like the Euclidean ones up to some controllable renormalizations. However, it is still not clear which balls to choose, and the lack of an invariance group for (\ref{Sigmak}) that allows renormalizations when $k < n$ makes this seem like a difficult route.





\section{Discussion}\label{Discussion}
To motivate Theorem \ref{Main1}, a useful starting point is to examine the distance function $d$ from a smooth minimal hypersurface $\Sigma \subset \mathbb{R}^n$. It is well-known that
\begin{equation}\label{DistFcn}
\Delta d \leq 0
\end{equation}
on both sides of $\Sigma$ (see e.g. \cite{CCo}). We claim that $d$ is a viscosity supersolution of (\ref{Sigma2}), despite having an ``upwards corner" on $\Sigma$.

To see this, assume that $\varphi \in C^2$ is $2$-convex and touches $d$ from below at $x_0$. The $2$-convexity of $\varphi$ implies that
\begin{equation}\label{Sigma2Basic}
\Delta \varphi \geq |D^2\varphi|.
\end{equation}
It follows immediately from (\ref{DistFcn}) and (\ref{Sigma2Basic}) that if $x_0 \notin \Sigma$, then $D^2\varphi(x_0) = 0$. The alternative is that $x_0 \in \Sigma$. Let $\nu$ be the unit normal to $\Sigma$ at $x_0$ and let $\Delta_{\Sigma}$ denote the Laplace operator on $\Sigma$. By the minimality of $\Sigma$ we have 
$$\Delta_{\Sigma}\varphi(x_0) = \Delta \varphi(x_0) - \varphi_{\nu\nu}(x_0).$$
Since $\varphi \leq 0$ on $\Sigma$ and $\varphi(x_0) = 0$, we have $\Delta_{\Sigma}\varphi(x_0) \leq 0$. We conclude using the previous identity and (\ref{Sigma2Basic}) that
$$0 \geq \Delta \varphi(x_0) - \varphi_{\nu\nu}(x_0) \geq |D^2\varphi(x_0)| - \varphi_{\nu\nu}(x_0).$$
It follows that at least $n-1$ eigenvalues of $D^2\varphi(x_0)$ vanish, hence that 
$$\sigma_2(D^2\varphi)(x_0) = 0 \leq 1$$ 
as desired.

\begin{rem}
The argument in fact shows that $\sigma_2(D^2d) \leq 0$ in the viscosity sense.
\end{rem}

In view of this fact, one strategy to produce a singular solution to (\ref{Sigma2}) is to build a viscosity subsolution to (\ref{Sigma2}) that lies beneath the distance function from some minimal surface $\Sigma$, vanishes on $\Sigma$, and also has an ``upwards singularity" (e.g. an upwards corner or Hessian unbounded from above) along $\Sigma$. This is no problem for a subsolution, since $C^2$ test functions could not touch from above at such points. One could then solve the Dirichlet problem for (\ref{Sigma2}) with boundary data in between that of said viscosity subsolution and the distance function from $\Sigma$, to produce a viscosity solution that (by the maximum principle) is trapped in between the two, and thus inherits a singularity along $\Sigma$. However, Theorem \ref{Main1} is an obstruction to carrying out this strategy.

\begin{rem} 
The idea of building exact solutions having some property by building appropriate super and subsolutions, solving the Dirichlet problem, and applying the maximum principle is pervasive in elliptic PDE. See e.g. \cite{M3} to see this idea carried out in the context of the Monge-Amp\`{e}re equation, and \cite{BDG} in the context of the minimal surface equation.
\end{rem}

As mentioned in the introduction, another idea towards building singular solutions to (\ref{Sigma2}), and more generally to (\ref{Sigmak}), is to model them on global one-homogeneous viscosity solutions to
\begin{equation}\label{Homogeneous}
\sigma_k(D^2u) = 0.
\end{equation}
The simplest examples have the form $u(x) = w(x'),$ where $x' \in \mathbb{R}^m,\, m < n$, and $w$ is one-homogeneous and smooth away from the origin. The reason we take $m < n$ is that we expect singularities to propagate to the boundary, by the convexity of $\{\sigma_k = 1\} \cap \Gamma_k$. As proof of concept, in the case $n \geq k \geq 3$, solutions to (\ref{Homogeneous}) of the form 
$$u(x) = |x'|, \quad x' \in \mathbb{R}^m, \quad k/2 < m  < k$$ 
all in fact arise as blowups by one-homogeneous rescaling of singular solutions to (\ref{Sigmak}) (see e.g. \cite{M4}, see also \cite{C}, \cite{M3}, \cite{CY}).

In the case $k = 2$, Theorems \ref{Main1} and \ref{Main2} suggest that none of these models arise in a simple way. The case $m = 1$ corresponds to the model $u = |x_1|$, which vanishes on a minimal surface. The cases $m = 2,\,3$ reduce to the case $m = 1$. Indeed, in both of these cases, $D^2w$ has at most one nonzero eigenvalue, by one-homogeneity when $m = 2$ and by one-homogeneity and the equation (\ref{Homogeneous}) when $m = 3$. Since $w$ is subharmonic we conclude that it is convex. If its subgradient image were at least two-dimensional, we could touch $w$ from below at $0$ by a convex quadratic polynomial with two positive Hessian eigenvalues, contradicting (\ref{Homogeneous}). Finally, in general, the model is in $W^{2,\,p}$ for all $p < m$, which when $m \geq 3$ appears ``too regular" to give rise to a singularity, according to Theorem \ref{Main2}. 

\begin{rem}
The functions
$$u(x',\,x'') = |x'|(1+|x''|^2), \quad x' \in \mathbb{R}^2, \quad x'' \in \mathbb{R}^{n-2}, \quad n \geq 3,$$
modeled on $|x'|$, solve
$$\frac{1}{2}\sigma_2(D^2u) = n-2 + (n-2)(n-3)|x'|^2 + (n-4)|x''|^2$$
away from $\{|x'| = 0\}$. The RHS is positive and analytic near zero. However, neither $u$ nor its model $|x'|$ are viscosity  supersolutions to the equations they solve classically away from $\{|x'| = 0\}$.
\end{rem}

\begin{rem}\label{Hedgehogs}
It is still interesting to understand what the one-homogeneous solutions to (\ref{Homogeneous}) can look like. One-homogeneous functions $u$ on $\mathbb{R}^n$ have a gradient image (``hedgehog") $\nabla u(\mathbb{S}^{n-1})$ with a second fundamental form $\mathrm{I\!I}$ satisfying
$$\mathrm{I\!I}(\nabla u(x)) = (D_T^2u)^{-1}(x)$$
at points $x \in \mathbb{S}^{n-1}$ where $D_T^2u$ is nonsingular. Here $D_T^2u$ is the Hessian of $u$ restricted to the tangent plane $T$ to $\mathbb{S}^{n-1}$ at $x$. If $u$ solves (\ref{Homogeneous}), we thus have
$$\frac{\sigma_{n-1-k}}{\sigma_{n-1}}(\mathrm{I\!I}) = 0, \quad \mathrm{I\!I}^{-1} \in \Gamma_{k}$$
at regular points of the hedgehog. The case $k = n-2$ is particularly inviting, since it corresponds to functions whose hedgehogs are minimal surfaces where they are regular. This corresponds to the case $k = 2,\, m = 4$ in the singularity models considered above. The information that $k \notin \text{rank}(D^2u)$ could play an important role in classification results for one-homogeneous solutions to (\ref{Homogeneous}), as it did in the case $k = 2,\,m = 3$ discussed above. 
\end{rem}

\begin{rem}
One-homogeneous functions have been studied in the context of singular solutions to fully nonlinear PDEs by many authors, see e.g. \cite{HNY}, \cite{M2}, and the references therein.
\end{rem}

To conclude we remark that beyond not being modeled in a simple way on the examples discussed above, any singular solution to (\ref{Sigma2}) would need to satisfy:
\begin{enumerate}
\item The singular set propagates to the boundary (convexity of $\{\sigma_2 = 1\} \cap \Gamma_2$),
\item The solution has unbounded Hessian from both sides in every neighborhood of a singular point (in view of the a priori estimate for semiconvex solutions in \cite{SY1}),
\item The solution has Hessian that satisfies the dynamic semiconvexity condition (1.2) from \cite{SY2} at some points and does not satisfy it at others in every neighborhood of a singular point. Here we are using that (\ref{Sigma2}) is uniformly elliptic at points where the dynamic semiconvexity condition is not satisfied.
\end{enumerate}
Thus, if a singular solution to (\ref{Sigma2}) exists, it is probably complicated, especially in comparison to the known singular solutions to (\ref{Sigmak}) in the case $n \geq k \geq 3$. 




\end{document}